\documentclass[12pt]{amsart} 

\usepackage{amssymb}
\usepackage[pdftex]{graphicx}

\newtheorem{theorem}{Theorem}[section]
\newtheorem{thm}[theorem]{Theorem}
\newtheorem{lemma}[theorem]{Lemma}

\theoremstyle{definition}

\newtheorem{example}[theorem]{Example}

\theoremstyle{remark}

\numberwithin{equation}{section}

\def\<{\langle}
\def\>{\rangle}

\def\R{\mathcal R}

\def\ln{\langle\!\langle}
\def\rn{\rangle\!\rangle}

\def\-{\underline}
\def\Z{\mathbb Z}
\def\N{\mathbb N}
\def\G{\Gamma}

\catcode`\@=11
\def\serieslogo@{\relax}
\def\@setcopyright{\relax}
\catcode`\@=12

\begin{document}

\title{Direct factors of profinite completions and decidability}
 
\author[Martin R. Bridson]{Martin R.~Bridson}
\address{Mathematical Institute, 24-29 St Giles', Oxford OX1 3LB, UK} 
\email{bridson@maths.ox.ac.uk} 

\date{19 March 2008}

\subjclass[2000]{20E18, 20F10}

\keywords{Profinite groups, direct factors, decidability}

\thanks{The author's research is funded 
by a Senior Fellowship from the EPSRC of Great Britain.}

\begin{abstract} We consider finitely presented,
residually finite groups $G$ and  
 finitely generated normal subgroups
 $A$ such that the inclusion $A\hookrightarrow G$
induces an isomorphism from the profinite completion of $A$ to a direct
factor of the profinite completion of $G$.
We explain why $A$ need not be a direct factor of a subgroup of finite index in $G$; indeed $G$
need not have a subgroup of finite index that splits as a non-trivial direct product. We prove that there 
is no algorithm that can determine whether $A$
is a direct factor of a subgroup of finite index in $G$.
\end{abstract}

\maketitle

Let $G$ be a finitely generated residually finite group. The inclusion $A\hookrightarrow G$
of any finitely generated subgroup
induces a morphism of profinite completions $\iota : \hat A\to \hat G$. If
$A$ is a direct factor of $G$ then $\iota$ is injective and we can identify the closure 
$\overline A$ of $\iota(A)$ with $\hat A$. In \cite{NS} Nikolov and Segal answered a 
question of Goldstein and Guralnick \cite{GG} by showing that the converse
of the preceding observation is false: there exist pairs of finitely generated residually
finite groups $A\hookrightarrow G$, with $A$ is normal in $G$, such
that $\iota : \hat A \to \overline A$ is an isomorphism, $\overline A$ is a direct
factor of $\hat G$, but $A$ is not a direct factor of $G$, nor indeed of any subgroup of
finite index in $G$.

Nikolov and Segal proved this by exhibiting an explicit group of the form $G=A\rtimes_\alpha\Z$, where
$A$ is finitely generated and $\alpha$, although not inner, induces an inner automorphism
on $A/N$ for every $\alpha$-invariant subgroup of finite index  $N\subset A$.

The first purpose of the present note is to explain how pairs of
residually finite groups $A\hookrightarrow G$ settling the Goldstein-Guralnick question
also arise from the constructions in \cite{BG}.  As well as providing a broader range of examples,
these constructions allow one to impose extra conditions on $A$ and $G$ (see subsection \ref{exs}).
 For example,
one can require $G$ to be finitely presented, indeed to be a direct product of torsion-free
hyperbolic groups and hence
have a finite classifying space. If one drops the requirement that $A$ be normal, one can
arrange for both $A$ and $G$ to be finitely presented.

Our main construction also provides a large classes of examples of the
Nikolov-Segal type $G=A\rtimes_\alpha\Z$; see subsection \ref{ss:NS}.

\smallskip

The second part of this note concerns the decision
problem associated to Goldstein and Guralnick's question:
a group $G$ 
 is given by a finite
presentation $G=\langle X\cup \mathcal A\mid R\>$;
it is guaranteed that $G$ is residually finite and that the subgroup $A\subset G$
generated by $\mathcal A$ has the property that the
inclusion map induces an isomorphism
from $\hat A$ to a direct factor of $\hat G$;
is there an algorithm that, given this data, can determine
whether or not
 $A$ is a direct factor of a subgroup of finite
index in $G$?

\begin{thm}\label{t:noAlgo} There does not exist
an algorithm that, given the above data,
can determine whether or not $A$ is a direct factor 
of  any subgroup of finite index in $G$.
\end{thm}

Theorem \ref{t:precise} below provides a
more precise formulation of this result.
   
\section{Building Examples}

In \cite{BG} Bridson and Grunewald settled a question of Grothendieck \cite{groth} by
constructing pairs of finitely presented, residually finite groups $j: P\hookrightarrow\Gamma$ such
that $\hat j :\hat P\to\hat\Gamma$ is an isomorphism but $P$ is not isomorphic (or even
quasi-isometric to) $\Gamma$. Pairs of finitely generated groups with this property had
been found earlier by Platonov and Tavkin \cite{PT}, Bass and Lubotzky \cite{BL},
and Pyber \cite{pyber}. A  simplified form of the arguments given
in \cite{BG} provides a flexible technique for constructing finitely generated
examples of a different type. The purpose of this section is to explain these
arguments and to apply them to the study of direct factors of profinite groups.

\subsection{The Basic Construction}

\begin{theorem}\label{main}
If $Q$ is a finitely presented group that is infinite but has
no non-trivial finite quotients, and if $H_2(Q,\Z)=0$, then there is a short exact sequence
of groups $1\to N \to \G\to Q\to 1$ where $\G$ is finitely presented and residually finite,
$N$ is finitely generated
but not finitely presentable, and  $N\to\G$ induces an isomorphism of profinite completions
$\hat N\to \hat\G$. Moreover there exists an algorithm that, given a finite presentation
of $Q$, will construct a finite presentation of $\Gamma$ and a generating set for $N$.
\end{theorem}

In order to make this theorem  useful one needs a supply of suitable
groups $Q$. This presents no difficulty since
one can embed any
finitely presented group $H$ in a finitely presented group $\overline H$
that has no finite quotients, matching the geometry and complexity of
$\overline H$ to that of $H$ in various ways \cite{mb-ep}, and the universal
central extension of $\overline H$ can then serve as $Q$ (see section 8
of \cite{BG}). Historically speaking, 
the first  suitable group $Q$ is
Higman's famous example $J=\langle a,b,c,d \mid aba^{-1}=b^2,\, bcb^{-1}=c^2,\, 
cdc^{-1}=d^2,\, dad^{-1}=d^2\rangle$.
Other small examples are constructed in \cite{BG}.

\medskip

Theorem \ref{main} is a consequence of the following two results. 

\begin{lemma}\label{PT} Let  $1\to N \to \G\to Q\to 1$ be a short exact 
sequence of finitely generated groups.
If $Q$ has no non-trivial finite quotients and $H_2(Q,\Z)=0$, then $N\to \G$
induces an isomorphism $\hat N\to \hat \G$.
\end{lemma}

\begin{proof} 
The surjectivity of $\hat N\to \hat \G$ is an immediate consequence of the observation
that since $Q$ has no finite quotients, if $f:\G\to F$ is a homomorphism onto a finite group then
$f(N)=F$. For injectivity, it is enough to consider finite index subgroups $I\subset N$ 
that are normal in $\G$ and to find a finite index subgroup  $H\subset \G$ that intersects
$N$ in $I$. The action of $\G$ by conjugation on $N$ induces a map to
the automorphism group of $N/I$, with kernel $M$ say. Since $M$ has finite index in 
$\G$, it maps onto $Q$ and 
we have a central extension
$$
1\to (N/I)\cap (M/I) \to M/I\to Q\to 1.
$$
Since $Q$ is super-perfect (that is, $H_1Q=H_2Q=0$), this extension splits (see \cite{milnor},
pp.~43-47). Setting
$H$ equal to the kernel of the resulting homomorphism $M\to M/I \to (N/I)\cap (M/I)$
completes the proof.
\end{proof}

The second ingredient in the proof of Theorem \ref{main} is  Wise's variation
on the Rips construction (see \cite{rips}, \cite{wise}).

\begin{theorem}\label{rips} There exists an
algorithm that, given a finite presentation of a group
$Q=\<X\mid R\>$, constructs a  finite presentation $\<X\cup\{\nu_1,\nu_2,\nu_3\}
\mid S\>$ for a torsion-free, residually finite group $\Gamma$,
of cohomological dimension 2, that is 
hyperbolic in the sense of Gromov.
The
subgroup $N\subset\G$ generated by
$\{\nu_1,\nu_2,\nu_3\}$ is normal but not free,
and $\G/N\cong Q$.
\end{theorem}  

The Rips-Wise algorithm, which is based on small-cancellation theory, is extremely explicit --- see
\cite{BG}, section 7.

The only assertion of Theorem \ref{main} that does
not follow immediately from Lemma \ref{PT} and Theorem \ref{rips} 
is the fact that $N$ is not finitely presentable. This is a special
case of Bieri's theorem that a finitely presentable normal subgroup of
a group of cohomological dimension 2 is either free or of finite index \cite{bieri}.

\subsection{In Answer to the Goldstein-Guralnick Question}\label{exs}

The following examples $A\hookrightarrow G$ provide a negative
answer to the question of Goldstein and Guralnick described
in the introduction. This question arose from a desire
to weaken the hypotheses in their generalisation \cite{GG} of Ayoub's splitting theorem \cite{ayo}.

\begin{example} Let $N$ and $\G$ be as in Theorem \ref{main} and let $B$ be
any finitely presented residually finite group. Let 
$G=\G\times B$ and let $A=N\times\{1\}$. Then $A\to G$ induces
an isomorphism $\hat A \to \overline A$ and $\hat G = \hat\G \times \hat B
= \overline A \times \hat B$. But $A$ is not (isomorphic to) a direct factor of $G$ or
any subgroup of finite index in $G$, since such direct factors are finitely presentable.
\end{example}

\smallskip

\begin{example} To obtain an example where $G$ and its subgroups of finite index have no non-trivial 
direct product decompositions whatsoever, one takes $B=\G$ in the above construction
and instead of defining $G$ to be  $\G\times\G$ one takes it to be
the fibre product $P\subset \G\times\G$ of the map
$\pi:\G\to Q$ in Theorem \ref{main}. More explicitly,
$P=\{(\gamma_1,\gamma_2)
\mid \pi(\gamma_1)=\pi(\gamma_2)\}$.
 It is proved in \cite{BBMS} that if $Q$  has
an Eilenberg-Maclane space $K(Q,1)$ with a finite 3-skeleton
(as Higman's group $J$ does, for example), then $P$ will be finitely presented. And it
is proved in \cite{BG} that $P\to\G\times\G$ induces an isomorphism of profinite
completions. By examining centralizers one can see that $P$ does not virtually split as a direct product
(see section 6 of \cite{BG}). 
It follows that the inclusion of $A=N\times\{1\}$ into $P$ maps $\hat A$ isomorphically  onto
the first factor of
 $\hat P = \hat \G\times \hat\G$, but neither $P$
nor any of its subgroups of finite index  decomposes as a non-trivial direct product.
\end{example}

\smallskip

\begin{example}  With $P$ and $\G$ as in the preceding example, $P\times\{1\}
\hookrightarrow (\Gamma\times\Gamma)\times
\Gamma$ provides  us with a pair of
{\em{finitely presented}} groups $A\hookrightarrow
G$ with $\hat A\cong \overline A$ a direct
factor of $\hat G$ but $A$ not a direct factor of
any subgroup of finite index in $G$. But in
this example $A$ is not normal in $G$.
\end{example}

\smallskip

\subsection{Examples of Nikolov-Segal Type
$G=A\rtimes\Z$}\label{ss:NS}

We continue with the notation established in
the proof of Lemma \ref{PT}, insisting now that $Q$ be infinite.
Since $Q$ has no finite quotients, the image of $\G$ in the automorphism
group of $N/I$ must coincide with that of $N$, in other words $\Gamma$ must act
on $N/I$ by inner automorphisms. Suppose that  $\G$
is torsion-free and hyperbolic, and note that $N$ cannot be 
cyclic as it has infinite index in its normalizer \cite{gromov}.
In this case, no element $\gamma
\in\Gamma\smallsetminus N$ can
act on $N$ (by conjugation in $\G$) as an inner automorphism, for if it acted
as conjugation by $n\in N$, say, then $n^{-1}\gamma$ would centralize $N$,
whereas non-cyclic subgroups of torsion-free hyperbolic groups have trivial
centralizers \cite{gromov}.

Thus, if $Q$ is torsion-free (Higman's group, for example) and $\G$ is
constructed as in Theorem \ref{rips}, then for every $\gamma\in \G\smallsetminus N$
we have $\langle N,\gamma \rangle = N\rtimes \langle \gamma \rangle \cong N\rtimes_\alpha\Z$
where no power of $\alpha$ is inner but $\alpha$ acts as an inner automorphism on
$N/I$ for every characteristic subgroup of finite index $I\subset N$ (and hence every
$\alpha$-invariant subgroup of finite index $K$, as one sees by considering the intersection
$I$ of all subgroups of index $|N/K|$). 

\section{The recognition problem for direct factors}

In this section we shall prove (a more precise version of) Theorem \ref{t:noAlgo}.
The seed of undecidability that we shall exploit in order to prove Theorem \ref{t:noAlgo}
comes from the following theorem, which is proved in \cite{PFdecide}.

\begin{thm}\label{t:triv} There exists a finitely
generated free group $F=F(X)$ and a recursive
sequence of finite subsets $\R_n\subset F$
so that there is no algorithm to
determine which of the groups $Q_n=F/\ln \R_n\rn$
is trivial, but each of the groups has the
following properties:
\begin{enumerate}
\item  $H_1(Q_n;\Z) = H_2(Q_n;\Z)=0$;
\item $Q_n$ has no non-trivial finite quotients.
\end{enumerate}
(If $Q_n\neq 1$ then $Q_n$ is infinite.)
\end{thm}

This theorem is proved in three stages. First one constructs a
sequence of finite group-presentations $\Pi_n\equiv \langle Y \mid \Sigma_n\>$
so that the groups presented are torsion-free and there is no algorithm
that can determine which are trivial. Secondly,
one modifies these presentations in an algorithmic manner so as to
ensure that none of the groups  presented has any proper subgroups of
finite index. Finally, an additional algorithm is implemented that replaces
$\Pi_n'$,
the modified $\Pi_n$, with a finite presentation $\tilde \Pi_n \equiv \<X\mid \R_n\>$
for the universal central extension of the group presented by $\Pi_n'$.
See \cite{PFdecide} for details.

\subsection{The Proof of Theorem \ref{t:noAlgo}}
Consider the sequence of pairs of groups
$N_n\hookrightarrow\G_n$ obtained by 
applying the Rips-Wise algorithm (Theorem \ref{rips}) to the presentations 
$\< X\mid \R_n\>$ from Theorem \ref{t:triv}. The output of the algorithm is a
recursive sequence of finite presentations
$\mathcal P_n \equiv \< X\sqcup\{\nu_1,\nu_2,\nu_3\}\mid S_n\>$
for $\G_n$, with
$N_n\subset\G_n$ given as
the subgroup generated by $\{\nu_1,\nu_2,\nu_3\}$.
Augmenting $\mathcal P_n$ with an additional generator $t$
and the relations $[t,x]=1$ for all $x\in X\cup\{\nu_1,\nu_2,\nu_3\}$
gives a recursive sequence of presentations $\mathcal P_n^+$,
for $G_n:=\Gamma_n\times\Z$ with $A_n:=N_n\times\{1\}$
the subgroup generated by $\{\nu_1,\nu_2,\nu_3\}$.

If $Q_n=1$ then $N_n=\G_n$. If $Q_n\neq 1$ then 
Theorem \ref{main} assures
us that the inclusion
$N_n\hookrightarrow \G_n$ still induces an isomorphism
 $\hat N_n\to\hat\G_n$, but $N_n$ is not finitely
presentable. Thus $A_n\hookrightarrow G_n$
always maps $\hat A_n$ isomorphically to the first
factor of $\hat G_n = \hat\G_n\times\hat\Z$, but
$A_n$ is a direct factor of $G_n$ (equivalently, some
subgroup of finite index in $G_n$) if and only
if $Q_n\neq 1$. And there is no algorithm that can determine
for which $n$ the group $Q_n$ is trivial.
\hfill $\square$

\smallskip
 
The version of Theorem  \ref{t:noAlgo} stated in the introduction was crafted so as
to be immediately comprehensible and free of technical jargon. We close with a
more technical statement that has greater precision. This is what is actually proved by
the preceding argument.

\begin{thm}\label{t:precise} There exists a finite set $Y=X\sqcup\{\nu_1,\nu_2,\nu_3,t\}$
and a recursive sequence  $(S_n)$ of finite sets of words in the letters $Y^{\pm 1}$
so that the groups $G_n:= F(Y)/\ln S_n\rn$, the subgroup $A_n\subset G_n$
generated by the image of $\{\nu_1,\nu_2,\nu_3\}$, and the subgroup
$B_n\subset G_n$ generated by the image of $\{t\}$, 
have the following properties:
\begin{enumerate}
\item each $G_n:= F(Y)/\ln S_n\rn$ is residually-finite and torsion-free;
\item each $B_n$ is infinite and $\hat G = \overline A_n\times \overline B_n$;
\item the inclusion $A_n\hookrightarrow G_n$ induces an isomorphism
$\hat A_n\to \overline A_n$;
\item the set $\{n \mid A_n \text{ is a direct factor of } G_n\}\subset\N$
is not recursive;
\item if $A_n$ is not a direct factor of $G_n$ then neither is it a direct factor of any subgroup of
finite index in $G_n$.
\end{enumerate}
\end{thm}

\end{document}